\documentclass[reqno,11pt]{amsart}

\usepackage{euscript}
\usepackage{enumerate}
\usepackage{amsmath}
\usepackage{mathrsfs}
\usepackage{mathtools}
\usepackage{verbatim}
\usepackage[initials]{amsrefs}
\usepackage{amssymb}
\usepackage{comment}
\usepackage{color}

\theoremstyle{plain}
\newtheorem{thm}{Theorem}[section]

\newtheorem{lemma}[thm]{Lemma}

\theoremstyle{remark}

\theoremstyle{definition}

\newtheorem{example}[thm]{Example}

 \newcommand\Cpx{{\mathbb C}}

 \newcommand\Mcal{{\mathcal{M}}}
 
\newcommand\mut{{\tilde{\mu}}}

 \newcommand\restrict{{\upharpoonright}}

\newcommand\DT{\operatorname{DT}}
\newcommand\eps{\epsilon}
\newcommand\Nats{{\mathbf N}}

\newcount\theTime
\newcount\theHour
\newcount\theMinute
\newcount\theMinuteTens
\newcount\theScratch
\theTime=\number\time
\theHour=\theTime
\divide\theHour by 60
\theScratch=\theHour
\multiply\theScratch by 60
\theMinute=\theTime
\advance\theMinute by -\theScratch
\theMinuteTens=\theMinute
\divide\theMinuteTens by 10
\theScratch=\theMinuteTens
\multiply\theScratch by 10
\advance\theMinute by -\theScratch

\def\today{{\number\day\space
 \ifcase\month\or
  January\or February\or March\or April\or May\or June\or
  July\or August\or September\or October\or November\or December\fi
 \space\number\year}}

\newcommand\HEu{{\EuScript H}}                   

\newcommand\supp{\operatorname{supp}}

\begin{document}

\title[]{Some Non-spectral DT-operators in finite von Neumann algebras, II: measures with atoms}

\author[Dykema]{Ken Dykema}
 \address{Ken Dykema, Department of Mathematics, Texas A\&M University, College Station, TX 77843-3368, USA.}
 \email{kdykema@math.tamu.edu}
%

\author[Krishnaswamy-Usha]{Amudhan Krishnaswamy-Usha}
\address{Amudhan Krishnaswamy-Usha}
\email{kuamudhan@gmail.com}



\begin{abstract} 
Results about angles between Haagerup--Schultz projections for DT-operators whose measures have atoms are proved, which in some cases imply that such operators are non-spectral.
Several examples are considered.
\end{abstract}

\date{\today}

\maketitle

\section{Introduction}

This note is a continuation of the paper~\cite{DKU}, to which we refer, together with~\cite{DKU21}, for a more  detailed introduction.
Here is a brief recap.
Consider a finite von Neumann algebra, namely, a von Neumann algebra $\Mcal\subseteq B(\HEu)$ that is equipped with a normal, faithful, tracial state $\tau$.
Given $T\in \Mcal$ and a Borel set $B\subseteq\Cpx$, the {\em Haagerup--Schultz projection}, $P(T,B)$, introduced in~\cite{HS09}, is an invariant projection for $T$,
with $P(T,B)\in\Mcal$, that is well behaved in connection with the Brown measure associated to the trace $\tau$.
An operator $T\in\Mcal$ is said to be {\em spectral} if it is similar (in $\Mcal$, as it happens) to the sum of a normal operator $N$ and a quasi-nilpotent operator $Q$ that commute.
The notion goes back to Dunford~\cite{D54}, where it is defined in terms of spectral decomposition.
In~\cite{DKU21}, the authors explored the relationship, for elements of $\Mcal$, between spectrality, decomposability and properties of their Haagerup--Schultz projections.
In particular, spectral operators must satisfy the {\em nonzero angle property} (NZA), which entails that for every Borel set, the angle, denoted $\alpha(P(T,B),P(T,B^c))$,
between the closed subspaces $P(T,B)\HEu$ and $P(T,B^c)\HEu$, is strictly positive.
Even more, spectral operators must satisfy the {\em uniformly nonzero angle property} (UNZA), which entails that
\[
\inf_{\{B\,:\,0<\mu(B)<1 \}} \alpha(P(T,B),P(T,B^c))>0,
\]
where the condition on $B$ is to ensure $P(T,B)\neq 0$ and  $P(T,B^c)\neq 0$.
Clearly, UNZA implies NZA.
We lack an example of an operator that is known to have the NZA but fails the UNZA.
Note that by Corollary~4.9 of~\cite{DKU21}, $T\in\Mcal$ is spectral if and only if it is decomposable and satisfies the UNZA property.

In~\cite{DKU21}, Voiculescu's circular operator and, more generally, circular free Poisson operators, were shown to fail the NZA and, thus, to be non-spectral.
These operators belong to the larger class of DT-operators (introduced in~\cite{DH04a}), that have natural models in upper triangular random matrices.
In~\cite{DKU}, different arguments were used to show that some other DT-operators fail the UNZA property and are non-spectral.
For every compactly supported Borel probability measure $\mu$ on $\Cpx$ and every $c>0$, by a $\DT(\mu,c)$-operator $Z$, we mean $Z\in\Mcal$ whose $*$-moments are
given by a specific formula in terms of $\mu$ and~$c$ (see~\cite{DH04a} or Section~4 of~\cite{DKU} for details).
Every DT-operator is decomposable, the Brown measure of a $\DT(\mu,c)$-operator is $\mu$ and the spectrum of the operator is the support of $\mu$ (see~\cite{DH04a}).
The case of Voiculsecu's circular operator corresponds to when $\mu$ is the uniform measure on the disk of radius $1$ (centered at the origin) and $c=1$, while the circular free Poisson
operators correspond to uniform distributions on annuli centered at the origin.
The operators treated in~\cite{DKU} were those whose measures $\mu$ are radially symmetric and have certain concentration properties.
While it is clear (from the fact that DT-operators are decomposable and Theorem~4.2 and Corollary~4.9 of~\cite{DKU21}) that every $\DT(\mu,c)$-operator whose measure $\mu$ is
supported on only finitely many points is spectral,
it is a open question whether all other DT-operators are non-spectral.

In this paper, we treat $\DT(\mu,c)$-operators whose measures $\mu$ have atoms.
Though our Lemma~\ref{lem:0A} is quite close to a result in ~\cite{DKU},
our main results (Theorems~\ref{thm:atom} and~\ref{thm:manyatoms}) don't follow from the results of~\cite{DKU}.
Finally, let us mention that our Examples~\ref{ex:AnotB} and~\ref{ex:no}, and also similar ones, are appealing objects of further investigation.

\section{Results}

For $0\le r<s<\infty$, $A(r,s)$ denotes the closed annulus (or disk, when $r=0$)
\[
A(r,s)=\{z\in\Cpx\mid r\le |z|\le s\}.
\]

What follows is an analogue of Lemma~6.2 of~\cite{DKU}.
Though it does not precisely follow from that lemma, its proof is nearly identical.
For completeness, we include its proof here.
\begin{lemma}\label{lem:0A}
Let $0<s'<s$ and $c>0$.  Let $Z$ be a $\DT(\mu,c)$-operator whose measure $\mu$ is concentrated in $\{0\}\cup A(s',s)$.
Let $t=\mu(\{0\})$ and suppose $0<t<1$.
Then 
\begin{multline}
\cos \left( \alpha(P(Z,\{0\}),P(Z,A(s',s)))\right)
\geq \left(1+\frac{s^2}{c^2\max(t,1-t)} \right)^{-1/2} \\
\geq \left(1+\frac{2s^2}{c^2} \right)^{-1/2}. 
\label{eq:aZlbd}
\end{multline}
\end{lemma}
\begin{proof}
Let $\mu_2$ denote the probability measure that is the renormalized restriction of $\mu$ to $A(s',s)$.
Then by Theorem~4.2 of~\cite{DKU}, there exists an example of a $\DT(\mu,c)$-operator
\begin{equation}\label{eq:Zmat}
Z=Z_1+pX(1-p)+Z_2= \left( \begin{matrix} Z_1 &  c p X (1-p) \\ 0 & Z_2 \end{matrix} \right),
\end{equation}
where
\begin{itemize}
\item[$\bullet$]
$p \in \Mcal$ is a projection with $\tau(p)=t$,
\item[$\bullet$]
$Z_1 \in p\Mcal p$ is, with respect to the trace $\tau_p=t^{-1}\tau\restrict_{p\Mcal p}$,
a $\DT(\delta_0,c\sqrt t)$-operator whose measure, $\delta_0$, is the Dirac mass at $0$,
\item[$\bullet$]
$Z_2 \in (1-p)\Mcal(1-p)$ is, with respect to the trace $\tau_{(1-p)}=(1-t)^{-1}\tau\restrict_{(1-p)\Mcal(1-p)}$,
a $\DT(\mu_2,c\sqrt{1-t})$-operator,
\item[$\bullet$]
$X\in\Mcal$ is a semicircular operator with $\tau(X^2)=1$
and so that $X$ and $\{Z_1,Z_2,p\}$ are $*$-free.
\end{itemize}
In the right most expression in~\eqref{eq:Zmat}, we are using the matrix notation
\[
\Mcal=\left(\begin{matrix} p\Mcal p & p\Mcal(1-p) \\ (1-p)\Mcal p & (1-p)\Mcal(1-p) \end{matrix}\right)_.
\]

Regard $Z$ as an operator acting on the Hilbert space $\HEu=L^2(\Mcal) =L^2(\Mcal,\tau)$.
For $x \in \Mcal$, let $\hat{x}$ denote the corresponding vector in $L^2(\Mcal)$. 
Since DT operators have spectra equal to the supports of their Brown measures (see Theorem 4.4 of~\cite{DKU}), $Z_2$ is invertible in $(1-p)\Mcal (1-p)$
and $Z_1$ is a quasinilpotent operator.
Since for every $\eps>0$, for all $k$ large enough we have $\|Z_1^k\|\le \eps^k$,
the series
\[
c \sum_{k=0}^n Z_1^k \left( p X (1-p) \right) Z_2^{-k-1} 
\]
converges in operator norm to some element $Y \in \Mcal$. Since $pZ_1=Z_1$ and $Z_2(1-p)=Z_2$, we have $Y=pY(1-p)$.
A direct computation shows that
\[ S=1+Y=\begin{pmatrix} p & Y \\ 0 & 1-p \end{pmatrix} \]
is invertible in $\Mcal$, with inverse
\[ S^{-1}=1-Y=\begin{pmatrix} p & -Y \\ 0 & 1-p \end{pmatrix}_. \]
Moreover, we get
\begin{equation}\label{eq:SZSinv}
S \begin{pmatrix} Z_1 & 0 \\  0 & Z_2 \end{pmatrix} S^{-1} =
\begin{pmatrix} Z_1 & YZ_2-Z_1Y \\ 0 & Z_2 \end{pmatrix}=
\begin{pmatrix} Z_1 & cpX(1-p) \\ 0 & Z_2 \end{pmatrix}=Z.
\end{equation}

Since $\sigma(Z_1)=\{0\}$ and $\sigma(Z_2)\subseteq A(s',s)$ (where the spectra are computed in the corners $p\Mcal p$ and $(1-p)\Mcal (1-p)$, respectively), we have 
\begin{gather*} 
P\left( \begin{pmatrix} Z_1 & 0 \\  0 & Z_2 \end{pmatrix}, \{0\} \right)\HEu = p\HEu, \\
P\left( \begin{pmatrix} Z_1 & 0 \\  0 & Z_2 \end{pmatrix}, A(s',s) \right)\HEu = (1-p)\HEu.
\end{gather*}
Using equation \eqref{eq:SZSinv} and Proposition~2.6(iii) of~\cite{DKU}, we get
\begin{equation}\label{eq:Arr'}
P(Z,\{0\}) \HEu = S \left(P\left( \begin{pmatrix} Z_1 & 0 \\  0 & Z_2 \end{pmatrix}, \{0\} \right) \HEu \right) = S (p\HEu) = p\HEu, 
\end{equation}
and
\begin{multline}
P(Z,A(s',s))\HEu = S \left( P\left( \begin{pmatrix} Z_1 & 0 \\  0 & Z_2 \end{pmatrix}, A(s',s) \right) \HEu \right)  = S((1-p)\HEu) \\
= \left \{ \begin{pmatrix} Y \eta  \\ \eta \end{pmatrix}: \eta \in (1-p)\HEu \right \}_.  \label{eq:As's}
\end{multline}
Let $\eta = \widehat{1-p}$. Then 
\[ Y\eta = c\sum_{k=0}^\infty Z_1^k(pX(1-p))Z_2^{-k-1} \widehat{1-p}_. \]
From the radial symmetry of the measure $\delta_0$, it follows that $Z_1$ and $\lambda Z_1$ have the same $*$-moments for every complex number $\lambda$ of modulus $1$.
Thus, $\tau((Z_1^j)^*Z_1^k)=0$ whenever $k$ and $j$ are nonnegative integers with $j\ne k$.
Now using $*$-freeness of $X$ and $\{Z_1,Z_2,p\}$, we calculate, for each $n\ge0$,
\begin{align*}
&\left\|c \sum_{k=0}^n Z_1^k(pX(1-p))Z_2^{-k-1}\widehat{(1-p)}\right\|_2^2 \notag \\ 
&=c^2 \sum_{0\le k_1,k_2\le n}\tau(\big(1-p)(Z_2^{-k_1-1})^*(1-p)Xp(Z_1^{k_1})^*Z_1^{k_2}pX(1-p)Z_2^{-k_2-1}\big) \notag \\
&=c^2 \sum_{0\le k_1,k_2\le n}\tau(1-p)\tau_{(1-p)}((Z_2^{-k_1-1})^*Z_2^{-k_2-1})\tau(X^2)\tau(p(Z_1^{k_1})^*Z_1^{k_2}p) \notag \\
&=c^2 \tau(p)\tau(1-p)\sum_{k=0}^n\tau_{(1-p)}((Z_2^{-k-1})^*Z_2^{-k-1})\tau_p((Z_1^k)^*Z_1^k) \\
&\ge c^2\tau(p)\tau(1-p)\tau_{(1-p)}((Z_2^{-1})^*Z_2^{-1}).
\end{align*}
From Corollary~5.5 of~\cite{DKU},
we have
\[\tau_{(1-p)} \left(  (Z_2^{*})^{-1} Z_2^{-1} \right) \geq s^{-2}. \]
Thus, we get
\begin{equation}
\| Y\eta \|_2^2 \geq c^2 t (1-t)s^{-2}. \label{eq:xiunder}
\end{equation}
From \eqref{eq:Arr'} and \eqref{eq:As's}, we have 
\[ \left( \begin{matrix}Y \eta \\ \eta \end{matrix} \right) \in P(Z,A(s',s)), \qquad \left( \begin{matrix} Y \eta \\ 0 \end{matrix} \right) \in P(Z,\{0\}). \]
Computing the cosine of the angle between $ \left( \begin{matrix}Y \eta \\ \eta \end{matrix} \right)$ and $ \left( \begin{matrix} Y \eta \\ 0 \end{matrix} \right)$
and using the lower bound~\eqref{eq:xiunder} gives 
\begin{multline}
\cos \left( \alpha(P(Z,\{0\}),P(Z,A(s',s)) \right) \geq \frac{ \| Y \eta \| }{((1-t)+\| Y \eta \|^2)^{1/2}} \\
=\left(1+\frac{1-t}{\|Y \eta\|^2}\right)^{-1/2} \geq \left(1+\frac{s^2}{c^2 t}\right)^{-1/2}_. \label{eq:cosalphat}
\end{multline}

In a similar fashion,
but now using the matrix notation
\[
\Mcal=\left(\begin{matrix}  (1-p)\Mcal(1-p) & (1-p)\Mcal p \\ p\Mcal(1-p) & p\Mcal p \end{matrix}\right)_,
\]
we find also an example of a $\DT(\mu,c)$ operator 
\[ Z = \left( \begin{matrix} Z_2 & c (1-p) X p \\ 0 & Z_1 \end{matrix} \right), \]
where
$p \in \Mcal$ is a projection with $\tau(1-p)=1-t$,  $Z_1 \in p\Mcal p$ is a $\DT(\delta_0,c \sqrt t)$-operator, $Z_2 \in (1-p)\Mcal(1-p)$ is a $\DT(\mu_2,c \sqrt{1-t})$-operator and $X$ is a semicircular operator with $\tau(X^2)=1$
and so that $X$ and $\{Z_1,Z_2,p\}$ are $*$-free.
With
\[
Y= -\sum_{k=0}^\infty Z_2^{-k-1}(c(1-p)Xp)Z_1^k
\]
and
\[
S = \begin{pmatrix} (1-p) & Y \\ 0 &  p \end{pmatrix}_,
\]
we have 
\[ 
S \begin{pmatrix} Z_2 & 0 \\ 0 & Z_1 \end{pmatrix} S^{-1} = Z.
\]
By similar reasoning to the above, we get
\[ \left( \begin{matrix}Y \widehat{p} \\ \widehat{p} \end{matrix} \right) \in P(Z,\{0\}), \qquad \left( \begin{matrix} Y \widehat{p} \\ 0 \end{matrix} \right) \in P(Z,A(s',s)). \]
Again using the rotational invariance of the $*$-moments of $Z_1$, we get the estimate
\[
\|Y\widehat{p}\|^2\ge c^2t(1-t)s^{-2}.
\]
This implies
\begin{equation}\label{eq:cosalpha1-t}
\cos \left( \alpha(P(Z,\{0\}),P(Z,A(s',s)) \right) \geq \left(1+\frac{s^2}{c^2 (1-t)}\right)^{-1/2}_.
\end{equation}
Combining the inequalities \eqref{eq:cosalphat} and \eqref{eq:cosalpha1-t} gives us the desired bound.
\end{proof}

\begin{thm}\label{thm:atom}
Let $\mu$ be a compactly supported Borel probability measure on $\Cpx$ and let $c>0$.
Let $Z$ be a $\DT(\mu,c)$ operator.
Suppose that  $\mu$ has an atom at a point that is an accumulation point of $\supp\mu$.
Then $Z$ fails to have the nonzero angle property (NZA) and is not a spectral operator.
\end{thm}
\begin{proof}
Since translating $\mu$ by a complex number $\lambda$ amounts to replacing $Z$ by $Z+\lambda I$, we may without loss of generality assume that $\mu$ has an atom at $0$
and that $0$ is an accumulation point of $\supp\mu$.
We will show
\begin{equation}\label{eq:PZ0}
\alpha(P(Z,\{0\}),P(Z,\Cpx\setminus\{0\}))=0,
\end{equation}
which gives immediately that $Z$ fails the NZA.
Since $0$ is an accumulation point of $\supp\mu$, there are $0<s_n'<s_n$ such that $\lim_{n\to\infty}s_n=0$ and such that for every $n$, $\mu(A(s_n',s_n))>0$.
For each $n$, let $B_n=\{0\}\cup A(s_n',s_n)$ and let $q_n=P(Z,B_n)$.
Let $\mut_n$ be the probability measure that is the renormalized restriction of $\mu$ to $B_n$.
Then, by Proposition 4.4(iii) of~\cite{DKU}, $Zq_n$ is a $\DT(\mut_n,c\sqrt{\mu(B_n)})$-operator.
Thus, $P(Zq_n,\Cpx\setminus\{0\})=P(Zq_n,A(s_n',s_n))$.
Using Lemma 6.1 of~\cite{DKU} and Lemma~\ref{lem:0A}, we get
\begin{multline*}
\cos(\alpha(P(Z,\{0\}),P(Z,\Cpx\setminus\{0\})) \\
\ge\cos(\alpha_{q_n\Mcal q_n}(P(Zq_n,\{0\}),P(Zq_n,\Cpx\setminus\{0\})))
\ge\left(1+\frac{2s_n^2}{c^2\mu(B_n)}\right)^{-1/2}.
\end{multline*}
Since $\mu(B_n)>\mu(\{0\})>0$ and $s_n\to0$, letting $n\to\infty$ proves~\eqref{eq:PZ0}.
\end{proof}

\begin{thm}\label{thm:manyatoms}
Let $\mu$ be a compactly supported Borel probability measure on $\Cpx$ and let $c>0$.
Let $Z$ be a $\DT(\mu,c)$ operator.
Suppose that $\mu$ has atoms at distinct points $a_1,a_2,\ldots\in\Cpx$.
Let $t_n=\mu(\{a_n\})$.
For each $n$, choose $d_n>0$ such that
\[
m_n:=\mu(\{z\in\Cpx\mid 0<|z-a_n|\le d_n\})>0.
\]
\begin{enumerate}[(a)]
\item\label{enum:dtm} If
\begin{equation}\label{eq:dmt}
\inf_{n\ge1}\frac{d_n^2}{m_n+t_n}=0,
\end{equation}
then $Z$ fails to have the uniformly nonzero angle property (UNZA) and is not a spectral operator.
\item If
\begin{equation}\label{eq:dt}
\inf_{n\ge1}\frac{d_n^2}{t_n}=0,
\end{equation}
then $Z$ fails to have the nonzero angle property (NZA) and is not a spectral operator.
\end{enumerate}
\end{thm}
\begin{proof}
Fix $n$ and let $d_n'>0$ be such that, letting
\[
B_n'=\{z\in\Cpx\mid d_n'<|z-a_n|\le d_n\},
\]
we have $\mu(B_n')\ge m_n/2$.
Let $q=P(Z,B_n'\cup\{a_n\})$.
Let $\mut$ be the renormalized restriction of $\mu$ to $B_n'\cup\{a_n\}$.
Then, by Proposition 4.4(iii) of~\cite{DKU}, $Zq$ is, with respect to the renormalized trace $\tau(q)^{-1}\tau\restrict_{q\Mcal q}$,
a $DT(\mut,c\sqrt{\mu(B_n')+t_n})$-operator.
Using Lemma 6.1 of~\cite{DKU} and (after translating by $a_n$) Lemma~\ref{lem:0A}, we get
\begin{multline*}
\cos(\alpha(P(Z,\{a_n\}),P(Z,\Cpx\setminus\{a_n\})) \\
\ge\cos(\alpha_{q\Mcal q}(P(Zq,\{a_n\}),P(Zq,\Cpx\setminus\{a_n\})))
\ge\left(1+\frac{2d_n^2}{c^2(\mu(B_n')+t_n)}\right)^{-1/2} \\
\ge\left(1+\frac{2d_n^2}{c^2((m_n/2)+t_n)}\right)^{-1/2}
\ge\left(1+\frac{4d_n^2}{c^2(m_n+t_n)}\right)^{-1/2}_.
\end{multline*}
Now the hypothesis~\eqref{eq:dmt} implies that
\[
\inf_n\alpha(P(Z,\{a_n\}),P(Z,\Cpx\setminus\{a_n\}))=0
\]
and, thus, that
$Z$ fails to have the UNZA property.

Now suppose~\eqref{eq:dt} holds.
By passing to a subsequence, we may assume without loss of generality that $d_n^2/t_n$ converges to $0$ and $a_n$ converges to a complex number as $n\to\infty$.
After translation, we may without loss of generality assume $a_n$ converges to $0$.
By passing to a subsequence, we may without loss of generality assume that all the $a_n$ lie in some sector of the complex plane having angle $\frac\pi3$.
For example, after rotation, we may without loss of generality assume $a_n=|a_n|e^{i\theta_n}$ with $0\le\theta_n\le \pi/3$ for all $n$.
Passing again to a subsequence, we may without loss of generality assume $|a_{n+1}|<|a_n|$, for all $n$.
By the law of cosines, 
\begin{align*}
|a_n-a_{n+1}|^2&=|a_n|^2+|a_{n+1}|^2-2|a_n|\,|a_{n+1}|\cos(\theta_n-\theta_{n+1}) \\
&\le|a_n|^2+|a_{n+1}|^2-2|a_n|\,|a_{n+1}|\cos(\frac\pi3) \\
&=|a_n|^2+|a_{n+1}|(|a_{n+1}|-|a_n|)<|a_n|^2.
\end{align*}
Since
\[
\mu(\{z\in\Cpx\mid0<|z-a_n|\le|a_n-a_{n+1}|\})\ge\mu(\{a_{n+1}\})>0,
\]
we may, by decreasing $d_n$ if necessary, take $d_n<|a_n|$ for all $n$.
Now by passing to a subsequence, we may without loss of generality assume 
\[
|a_{n+1}|<|a_n|-d_n
\]
for all $n$, which also implies
\[
|a_{n+k}|<|a_n|-d_n
\]
for all $n,k\ge1$.
In particular, we have
\begin{equation}\label{eq:anank}
|a_n-a_{n+k}|\ge|a_n|-|a_{n+k}|>d_n>d_{n+k}
\end{equation}
for all $n,k\ge1$.
Let
\begin{align*}
A&=\{a_n\mid n\in\Nats\} \\
B&=\bigcup_{n=1}^\infty\{z\in\Cpx\mid0<|z-a_n|<d_n\}.
\end{align*}
From~\eqref{eq:anank}, it follows that $A$ and $B$ are disjoint sets.
We will show
\begin{equation}\label{eq:PZAB}
\alpha(P(Z,A),P(Z,B))=0,
\end{equation}
which will imply $\alpha(P(Z,A),P(Z,A^c))=0$ and, thus, that $Z$ fails to have the NZA.
We argue as in the proof of the first part of the theorem.
Fixing $n$, we choose $d_n'>0$ such that, letting
\[
B_n'=\{z\in\Cpx\mid d_n'<|z-a_n|\le d_n\},
\]
we have $\mu(B_n')>0$.
Let $q=P(Z,B_n'\cup\{a_n\})$
and let $\mut$ be the renormalized restriction of $\mu$ to $B_n'\cup\{a_n\}$.
Then, by Proposition 4.4(iii) of~\cite{DKU}, $Zq$ is, with respect to the renormalized trace $\tau(q)^{-1}\tau\restrict_{q\Mcal q}$,
a $DT(\mut,c\sqrt{\mu(B_n')+t_n})$-operator.
Using Lemma 6.1 of~\cite{DKU} and (after translating by $a_n$) Lemma~\ref{lem:0A}, we get
\begin{multline*}
\cos(\alpha(P(Z,A),P(Z,B))
\ge\cos(\alpha(P(Z,\{a_n\}),P(Z,B_n'))) \\
\ge\cos(\alpha_{q\Mcal q}(P(Zq,\{a_n\}),P(Zq,\Cpx\setminus\{a_n\}))) \\
\ge\left(1+\frac{2d_n^2}{c^2(\mu(B_n')+t_n)}\right)^{-1/2}
\ge\left(1+\frac{2d_n^2}{c^2t_n}\right)^{-1/2}.
\end{multline*}
Letting $n\to\infty$ and using~\eqref{eq:dt} proves~\eqref{eq:PZAB}.

In both cases, failing either the UNZA or the NZA, $Z$ cannot be a spectral operator.
\end{proof}

\begin{example}\label{ex:yes}
Here is a class of probability measures $\mu$, each with countable support and
satisfying the hypotheses of Theorem~\ref{thm:manyatoms}(b), while
failing to satisfy the hypotheses of Theorem~\ref{thm:atom}.
Thus, the corresponding $DT(\mu,c)$-operators all fail to satisfy the NZA and are non-spectral.

Take the probability measure
\[
\mu=C_p\sum_{n=1}^\infty \frac1{n^p}\delta_{1/n},
\]
for $p>1$ and the suitable constant $C_p$.
Then $0$ is the only accumulation point of $\supp(\mu)$ and $\mu$ has no atom at $0$, so the hypotheses of Theorem~\ref{thm:atom} fail to hold.

However, in the notation of Theorem~\ref{thm:manyatoms},
letting $a_n=1/n$ and taking
\[
d_n=\frac1n-\frac1{n+1}=\frac1{n(n+1)},
\]
we see that $m_n>0$.
Of course, we have
$t_n=\frac{C_p}{n^p}$
and
\[
\frac{d_n^2}{t_n}\le\frac{n^{p-4}}{C_p}.
\]
If $1<p<4$, then the hypotheses of Theorem~\ref{thm:manyatoms}(b) are satisfied.
\end{example}

\begin{example}\label{ex:AnotB}
The following probability measure satisfies the hypothesis of Theorem~\ref{thm:manyatoms}(a) but fails to satisfy the hypotheses of Theorem~\ref{thm:manyatoms}(b) and Theorem~\ref{thm:atom}. The corresponding $DT(\mu,c)$-operator therefore fails to satisfy the UNZA property and is not spectral, but we do not know whether it has the NZA property.

Let $a_n=\sum_{k=1}^n \frac{\log k}{k^2}$, $t_n = C n^{-2}$ with $C=6/\pi^2$. Consider the measure
\[
\mu = \sum_{n=1}^\infty t_n \delta_{a_n}.
\]
Note that $\sum_{k=1}^\infty \frac{\log k}{k^2}$ is the only accumulation point of  $\supp(\mu)$, but $\mu$ has no atom there, so Theorem~\ref{thm:atom} does not apply in this case.
In the notation of Theorem~\ref{thm:manyatoms}, we need $d_n \geq a_{n+1}-a_n = \frac{\log(n+1)}{n+1}$. Since $t_n= C n^{-2}$,
\[
\frac{d_n^2}{t_n} \geq\frac{(\log (n+1))^2 n^2}{C(n+1)^2},
\]
which is clearly bounded away from zero, so Theorem~\ref{thm:manyatoms}(b) fails to hold.

If we take
\[
d_n=\sum_{k=n+1}^\infty \frac{\log k}{k^2},
\]
then  $\{a_k: k\geq n \}$ is a subset of the ball of radius $d_n$ around $a_n$, and therefore
\[
m_n + t_n \geq  \sum_{k=n}^\infty t_k = \sum_{k=n}^\infty \frac{C}{k^2} \geq \frac C n .
\]
Since $n^{-2} \log n $ is decreasing for $n\geq 3$, we get the bound
\[
d_n\le\int_n^\infty\frac{\log x}{x^2}\,dx=\frac{1+\log n}n.
\]
Therefore,
\[
\frac{d_n^2}{m_n+t_n}  \leq\frac{(1+\log n)^2}{Cn}
\]
which tends to $0$ as $n\to\infty$.
Hence, the corresponding $DT(\mu,c)$-operator fails to have the UNZA property, and is not spectral.
\end{example}

\begin{example}\label{ex:no}
Here is a a measure $\mu$ that has infinitely many atoms but which fails to satisfy the hypotheses of Theorem~\ref{thm:atom} and of Theorem~\ref{thm:manyatoms}.
We don't know whether the corresponding $DT(\mu,c)$-operator has the UNZA or the NZA properties,
nor do we know whether it is spectral.

Let 
\[
\mu=\sum_{n=1}^\infty 2^{-n}\delta_{1/n}.
\]
Then $0$ is the only accumulation point of $\supp(\mu)$ and $\mu$ has no atom at $0$, so the hypotheses of Theorem~\ref{thm:atom} fail to hold.

Using the notation of Theorem~\ref{thm:manyatoms},
we let $a_n=1/n$.
In order to have $m_n>0$, we need
\begin{equation}\label{eq:dnmin}
d_n\ge\frac1{n(n+1)}.
\end{equation}
Suppose
\[
d_n=t-\frac1n,\quad\frac1k\le t<\frac1{k-1},
\]
for an integer $k\in\{1,2,\ldots, n\}$.
We will show that all possible choices of values for $d_n$, in finitely many different cases, yield lower bounds for
$\frac{d_n^2}{m_n+t_n}$ that stay away from $0$ as $n\to\infty$.

If $k=n$, then using~\eqref{eq:dnmin} we have $m_n=2^{-n-1}$ and
\[
\frac{d_n^2}{m_n+t_n}\ge\frac{2^{n+1}}{3n^2(n+1)^2},
\]
which tends to $+\infty$ as $n\to\infty$.

If $k\in\{1,\ldots,n-1\}$, then
\[
d_n\ge\frac1k-\frac1n=\frac{n-k}{nk}
\]
and
\[
m_n+t_n\le\sum_{j=k}^\infty2^{-j}=2^{-k+1}.
\]
Thus,
\[
\frac{d_n^2}{m_n+t_n}\ge\left(\frac{n-k}{nk}\right)^22^{k-1}=\frac1{2n^2}f_n(k)^2,
\]
where we let
\[
f_n(x)=\left(\frac{n-x}x\right)2^{x/2}.
\]
We are interested in finding the minimum value of $f_n$ on the interval $[1,n-1]$.
Using
\[
f_n'(x)=-\left(\frac{\log 2}2\right)\frac{2^{x/2}}{x^2}\left(x^2-nx+\frac{2n}{\log 2}\right),
\]
we are interested in the endpoints $x=1$ and $x=n-1$, as well as the critical points
\[
x=r_n:=\frac{n-\sqrt{n^2-\frac{8n}{\log2}}}2,\qquad x=s_n:=\frac{n+\sqrt{n^2-\frac{8n}{\log2}}}2.
\]
We compute
\[
\lim_{n\to\infty}r_n=\frac2{\log 2}
\]
and deduce that for sufficiently large $n$, both $r_n$ and $s_n=n-r_n$ lie in the interval $[1,n-1]$.
We also compute:
\begin{align*}
\lim_{n\to\infty}\frac1{2n^2}f(1)^2&=1, \\
\lim_{n\to\infty}\frac1{2n^2}f(r_n)^2&=\lim_{n\to\infty}\frac1{2n^2}\left(\frac{s_n}{r_n}\right)^22^{r_n}=\frac{(\log2)^2}8 2^{2/\log2}>0, \\
\lim_{n\to\infty}\frac1{2n^2}f(s_n)^2&=\lim_{n\to\infty}\frac1{2n^2}\left(\frac{r_n}{s_n}\right)^22^{s_n}=+\infty, \\
\lim_{n\to\infty}\frac1{2n^2}f(n-1)^2&=+\infty.
\end{align*}
This ensures that for all allowed choices of $d_n$,
\[
\liminf_{n\to\infty}\frac{d_n^2}{m_n+t_n}>0.
\]
Thus, the hypotheses of Theorem~\ref{thm:manyatoms} fail to hold.
\end{example}

\end{document}